\documentclass[11pt]{article}

\usepackage[T1]{fontenc}
\usepackage[utf8]{inputenc}
\usepackage{lmodern}
\usepackage{amsmath,amssymb,amsthm,mathtools}
\usepackage{geometry}
\usepackage{microtype}
\usepackage[hyperfootnotes=false]{hyperref}
\usepackage{enumitem}
\usepackage{cite}

\hypersetup{colorlinks=true,linkcolor=blue,citecolor=blue,urlcolor=blue}
\newtheorem{theorem}{Theorem}[section]
\newtheorem{proposition}[theorem]{Proposition}
\newtheorem{lemma}[theorem]{Lemma}
\newtheorem{corollary}[theorem]{Corollary}
\theoremstyle{definition}
\newtheorem{definition}[theorem]{Definition}
\newtheorem{remark}[theorem]{Remark}

\newcommand{\C}{\mathbb C}
\newcommand{\R}{\mathbb R}
\newcommand{\Z}{\mathbb Z}
\newcommand{\N}{\mathbb N}
\newcommand{\T}{\mathbb T}
\newcommand{\dd}{\,\mathrm d}
\newcommand{\Ree}{\operatorname{Re}}
\newcommand{\Ims}{\operatorname{Im}}

\title{\textbf{Zero-Density Concentration for Dirichlet Polynomials}}
\author{
Eric Dubon\\[0.3em]
\href{https://orcid.org/0000-0003-3400-7756}
{ORCID: 0000-0003-3400-7756}
}
\date{September 2026}

\begin{document}
\maketitle

\begingroup
\renewcommand{\thefootnote}{}
\footnotetext{%
Department of Mathematics, University of Alicante,
03080 Alicante, Spain.
E-mail: \texttt{eric.dubon@ua.es}}
\addtocounter{footnote}{-1}
\endgroup

\begin{abstract}
We prove that broad families of finite Dirichlet sums exhibit deterministic concentration of normalized vertical zero density on a single line.  The main result gives a general criterion under which the normalized Jessen potentials converge locally uniformly to a piecewise-linear convex function with a single corner.  The associated normalized measures describing the distribution of real parts of zeros in vertical mean density then converge weakly to a Dirac mass.  Thus, for each fixed truncation, zeros may occupy a nontrivial range of real parts, while asymptotically their normalized vertical density concentrates on one line.

The proof combines the finite Bohr lift with a translation-uniform anti-concentration estimate for isolated prime coordinates.  We apply the criterion to partial sums of the Riemann zeta function, fixed Dirichlet $L$-functions, and primitive holomorphic Hecke eigenforms of fixed level, trivial Dirichlet character, and without complex multiplication.  The concentration line is $\Ree s=1/2$ in the normalized setting and becomes $\Ree s=k/2$ for the classical Fourier coefficients of a form of weight $k$.
\end{abstract}



\section{Introduction}

Let
\begin{equation}\label{eq:intro-P}
P_N(s)=\sum_{n\le N}a_n n^{-s},\qquad s=\sigma+it,
\end{equation}
be a finite Dirichlet sum with $a_1\neq0$.  For each fixed $N$, $P_N$ is an exponential polynomial in the vertical variable.  Its zeros may occupy a wide range of real parts, and a substantial literature has developed around their location, existence, density, and zero-free regions.

For the special case
\begin{equation}\label{eq:zetaN-intro}
\zeta_N(s)=\sum_{n\le N}n^{-s},
\end{equation}
Montgomery studied the extreme right-hand zeros \cite{Montgomery1983}; Borwein, Fee, Ferguson and van der Waall gave extensive numerical and structural information \cite{BorweinEtAl2007}; Gonek and Ledoan obtained zero-counting results and information on the distribution of real parts in finite-height rectangles \cite{GonekLedoan2010}; and Mora proved that, for every sufficiently large fixed $N$, the set of real parts of the zeros is dense throughout the corresponding critical strip \cite{Mora2013}.  These results show in particular that the support of the zero set of a fixed truncation can be much wider than any single distinguished vertical line.

The same circle of questions has been pursued for broader arithmetic Dirichlet polynomials.  Ledoan, Roy and Zaharescu studied zero-free regions and zero counts for partial sums of Dedekind zeta functions of cyclotomic fields \cite{LedoanRoyZaharescu2014}.  Roy and Vatwani developed zero-free regions and zero-density estimates for partial sums attached to multiplicative functions \cite{RoyVatwani2019}, and later proved sharp existence results for zeros of the corresponding Dirichlet polynomials \cite{RoyVatwani2021}.  In a complementary direction, the author used Bohr equivalence to prove density of attainable real parts for partial sums of the Dirichlet lambda, beta and eta functions \cite{Dubon2025}.  Very recently, Kerr, Klurman and Thorner \cite{KerrKlurmanThorner2026} proved that unusually large partial sums of Dirichlet coefficients can force low-lying zeros of the underlying infinite $L$-function away from the critical line.  Their zeros are zeros of $L(s,\pi)$ itself, whereas the present paper concerns zeros of the finite Dirichlet truncations.

There is also related work for modular forms with a different approximating object.  Li, Roy and Zaharescu considered symmetrized approximations to Hecke $L$-functions obtained by adjoining to a truncated Dirichlet series the dual term supplied by the functional equation \cite{LiRoyZaharescu2016}.  Their approximants are structurally different from the raw truncations considered here, and their asymptotic regime couples the truncation parameter with the height.

The question addressed in this paper is instead the following: \emph{where does the normalized vertical density of zeros concentrate as the truncation length tends to infinity?}  The distinction between support and mass is essential.  A family can have zeros with real parts dense throughout a large strip while an asymptotically negligible proportion of its vertical zero density lies away from one distinguished line.

For fixed $N$, the relevant convex potential is the Jessen function
\begin{equation}\label{eq:Jessen-intro}
J_{P_N}(\sigma)
=
\lim_{T\to\infty}\frac1{2T}
\int_{-T}^{T}\log|P_N(\sigma+it)|\dd t.
\end{equation}
The existence of this mean follows from Jessen and Tornehave \cite[Theorem~5]{JessenTornehave1945}, while its convexity follows from \cite[Theorem~7]{JessenTornehave1945}. Their Jensen formula for ordinary Dirichlet series\cite[Theorem~31]{JessenTornehave1945} identifies $\frac{1}{2\pi}J_{P_N}''$, in the sense of distributions, with the vertical zero-frequency measure. Recent work of Andersson \cite{Andersson2026} develops related mean-counting formulae for vertical limits of
infinite Dirichlet series.  Our setting is finite-dimensional and introduces the asymptotic parameter $N\to\infty$.

The novelty here is a limit theorem for the whole Jessen potential.  Under an explicit isolated-prime hypothesis, the normalized potentials converge locally uniformly to a piecewise-linear function with one corner; the corresponding normalized zero-density measures converge to a single Dirac mass.  For zeta and fixed Dirichlet $L$-functions the limiting line is $\Ree s=1/2$.  The same line occurs for normalized holomorphic Hecke eigenvalues, while the classical Fourier-coefficient normalization moves it to $\Ree s=k/2$.

The principal analytic input is a translation-uniform logarithmic anti-concentration estimate for sums of independent Steinhaus variables with comparable weights.  More precisely, if $Z_1,\ldots,Z_m$ are independent Steinhaus variables and $b_1,\ldots,b_m>0$ satisfy $\max_j b_j/\min_j b_j\le K$, then, uniformly in $a\in\C$,
\[
\mathbb E\left(\log\left|a+\sum_{j=1}^m b_jZ_j\right|\right)
\ge
\frac12\log\sum_{j=1}^m b_j^2-C_K.
\]
The uniformity in the translation is essential: after conditioning on the remaining coordinates of the finite Bohr torus, the isolated prime-coordinate block is translated by a complex value depending on those conditioned coordinates.

All density statements below use the Jessen mean-motion order of limits: for each fixed $N$ the height $T$ tends to infinity first, and only afterwards $N\to\infty$.  No coupled $(N,T)$ limit is asserted.  This is different from the finite-height regime studied, for example, by Gonek and Ledoan \cite{GonekLedoan2010}.

The paper is organized as follows. Section 2 states the abstract zero-density concentration criterion and its three arithmetic applications. Sections 3–5 develop the Bohr–Jessen framework, prove the translation-uniform anti-concentration estimate, and establish the abstract criterion. Sections 6–8 verify its hypotheses for zeta partial sums, fixed Dirichlet \(L\)-functions, and holomorphic Hecke eigenforms, respectively.
\section{Main results}\label{sec:main-results}

Let
\begin{equation}\label{eq:MN-intro}
M_N:=\max\{n\le N:a_n\neq0\}.
\end{equation}
When $a_1\neq0$, Proposition~\ref{prop:total-mass} below shows that $J_{P_N}''$ has total mass $\log M_N$.  Recall also that a finite convex function on an open interval has monotone one-sided derivatives \cite[Theorem~24.1]{Rockafellar1970}; consequently its distributional second derivative is a positive Radon measure, namely the Lebesgue--Stieltjes measure associated with its increasing derivative.

\begin{definition}[Zero-density concentration]\label{def:concentration}
Assume $M_N\to\infty$.  We say that the family $\{P_N\}$ exhibits \emph{zero-density concentration} on the line $\Ree s=\alpha$ if
\begin{equation}\label{eq:concentration-def}
\nu_N:=\frac{1}{\log M_N}J_{P_N}''
\xrightarrow{\mathrm w} \delta_\alpha
\end{equation}
weakly as probability measures on $\R$.
\end{definition}

This definition has a direct zero-counting interpretation.
By Jessen and Tornehave \cite[Theorem~31]{JessenTornehave1945},
for every $a<b$,
\begin{equation}\label{eq:JT-main-intro}
\frac{J_{P_N}''((a,b))}{2\pi}
=
\lim_{T\to\infty}\frac1{2T}
\#\{\rho=\beta+i\gamma:P_N(\rho)=0,\ 
a<\beta<b,\ |\gamma|<T\},
\end{equation}
where zeros are counted with multiplicity.
Thus \eqref{eq:concentration-def} means that asymptotically all normalized vertical zero density lies in every fixed neighborhood of the line $\Ree s=\alpha$.

For $P_N$ as in \eqref{eq:intro-P}, normalize $a_1=1$ and set
\begin{equation}\label{eq:global-energy-intro}
E_N(\sigma):=\sum_{n\le N}|a_n|^2n^{-2\sigma}.
\end{equation}

\begin{theorem}[Isolated-prime zero-density criterion]\label{thm:abstract-main}
Let $P_N(s)=\sum_{n\le N}a_n n^{-s}$ with $a_1=1$, and let $M_N$ be given by \eqref{eq:MN-intro}.  Suppose there exist $\alpha\in\R$ and subsets
\[
\mathcal Q_N\subset\{p\text{ prime}:N/2<p\le N\}
\]
with $\#\mathcal Q_N\to\infty$ such that:

\begin{enumerate}[label=\textup{(H\arabic*)}]
\item\label{H1intro} for every $\sigma\in\R$,
\begin{equation}\label{eq:H1-intro}
\frac{1}{2\log N}\log E_N(\sigma)
\longrightarrow(\alpha-\sigma)_+;
\end{equation}

\item\label{H2intro} for every compact interval $I\subset(-\infty,\alpha)$ there is $K_I\ge1$ such that, for all sufficiently large $N$, all $\sigma\in I$, and all $p,q\in\mathcal Q_N$,
\begin{equation}\label{eq:H2comp-intro}
K_I^{-1}
\le
\frac{|a_p|p^{-\sigma}}{|a_q|q^{-\sigma}}
\le K_I,
\end{equation}
and, with
\begin{equation}\label{eq:isolated-energy-intro}
B_N(\sigma):=\sum_{p\in\mathcal Q_N}|a_p|^2p^{-2\sigma},
\end{equation}
one has
\begin{equation}\label{eq:H2energy-intro}
\frac{1}{2\log N}\log B_N(\sigma)
\longrightarrow\alpha-\sigma
\end{equation}
locally uniformly on $(-\infty,\alpha)$.
\end{enumerate}

Then
\begin{equation}\label{eq:abstract-potential-intro}
\frac{J_{P_N}(\sigma)}{\log N}
\longrightarrow
(\alpha-\sigma)_+
\end{equation}
locally uniformly on $\R$, and
\begin{equation}\label{eq:abstract-measure-intro}
\frac1{\log M_N}J_{P_N}''
\xrightarrow{\mathrm w}\delta_\alpha.
\end{equation}
In particular, the normalized vertical zero density concentrates on the line $\Ree s=\alpha$.
\end{theorem}

\begin{remark}\label{rem:further-families}
The criterion is not intrinsically restricted to the arithmetic families considered below.  Its application to further families reduces to two arithmetic inputs: the quadratic-energy asymptotic in hypothesis \ref{H1intro}, and a sufficiently large family of isolated primes on which the normalized prime coefficients are bounded above and away from zero, as required in hypothesis \ref{H2intro}.  A suitable Sato--Tate-type equidistribution theorem provides a natural mechanism for the latter condition.
\end{remark}

\begin{theorem}[Riemann zeta partial sums]\label{thm:zeta-intro}
For
\[
\zeta_N(s)=\sum_{n\le N}n^{-s},
\]
let $J_N$ denote its Jessen function.  Then
\begin{equation}\label{eq:zeta-potential-intro}
\frac{J_N(\sigma)}{\log N}
\longrightarrow
\left(\frac12-\sigma\right)_+
\end{equation}
locally uniformly on $\R$, and
\begin{equation}\label{eq:zeta-measure-intro}
\frac1{\log N}J_N''\xrightarrow{\mathrm w}\delta_{1/2}.
\end{equation}
\end{theorem}

\begin{theorem}[Dirichlet $L$-functions]\label{thm:dirichlet-intro}
Let $\chi$ be a fixed Dirichlet character modulo $q$, and put
\[
L_N(s,\chi)=\sum_{n\le N}\chi(n)n^{-s}.
\]
Let $J_{N,\chi}$ be its Jessen function and
\[
M_{N,\chi}=\max\{n\le N:(n,q)=1\}.
\]
Then
\begin{equation}\label{eq:dir-potential-intro}
\frac{J_{N,\chi}(\sigma)}{\log N}
\longrightarrow
\left(\frac12-\sigma\right)_+
\end{equation}
locally uniformly on $\R$, and
\begin{equation}\label{eq:dir-measure-intro}
\frac1{\log M_{N,\chi}}J_{N,\chi}''
\xrightarrow{\mathrm w}\delta_{1/2}.
\end{equation}
\end{theorem}

For the modular application, let $f$ be a fixed normalized primitive
holomorphic Hecke eigenform of even weight $k\ge2$, fixed level $Q$,
and trivial Dirichlet character; see, for example,
Iwaniec \cite[Sections~6.6--6.8]{IwaniecTopics} for the theory and
normalization of newforms.  Thus $f$ has a Fourier expansion at the
cusp infinity of the form
\begin{equation}\label{eq:modular-fourier-intro}
f(z)=\sum_{n\ge1}a_f(n)e^{2\pi inz},
\qquad a_f(1)=1,
\end{equation}
where $a_f(n)$ denotes the $n$th Fourier coefficient of $f$.

We further assume that $f$ has no complex multiplication.  This is the
hypothesis under which the Sato--Tate equidistribution theorem used
below applies; see Barnet-Lamb, Gee and Geraghty
\cite[Corollary~7.1.7]{BLGG2011}.

Define the normalized Hecke eigenvalues by
\begin{equation}\label{eq:lambda-normalization-intro}
\lambda_f(n):=a_f(n)n^{-(k-1)/2}.
\end{equation}
With this normalization, the associated $L$-function is centered at
$\Ree s=1/2$.

\begin{theorem}[Fixed-level holomorphic Hecke eigenforms]\label{thm:modular-intro}
With $f$ as above, define
\begin{equation}\label{eq:automorphic-partial-intro}
L_N(s,f)=\sum_{n\le N}\lambda_f(n)n^{-s}.
\end{equation}
Then the Jessen functions $J_{N,f}$ satisfy
\begin{equation}\label{eq:modular-potential-intro}
\frac{J_{N,f}(\sigma)}{\log N}
\longrightarrow
\left(\frac12-\sigma\right)_+
\end{equation}
locally uniformly on $\R$, and
\begin{equation}\label{eq:modular-measure-intro}
\frac1{\log M_{N,f}}J_{N,f}''
\xrightarrow{\mathrm w}\delta_{1/2},
\end{equation}
where $M_{N,f}:=\max\{n\le N:\lambda_f(n)\neq0\}$.

Equivalently, for the classical Fourier-coefficient truncations
\begin{equation}\label{eq:classical-modular-intro}
\mathcal L_N(s,f)=\sum_{n\le N}a_f(n)n^{-s},
\end{equation}
zero-density concentration occurs on the line
\begin{equation}\label{eq:kover2-intro}
\Ree s=\frac{k}{2}.
\end{equation}
\end{theorem}

\section{Bohr--Jessen framework and zero-density measures}\label{sec:bohr-jessen}

\subsection{Finite Bohr lifts}

Fix $N\ge2$, let $p_1<\cdots<p_d\le N$ be the primes up to $N$, $d=\pi(N)$ and write
\begin{equation}\label{eq:prime-factorization}
n=\prod_{j=1}^d p_j^{v_{p_j}(n)},
\qquad
\mathbf v(n):=(v_{p_1}(n),\ldots,v_{p_d}(n))\in\N_0^d.
\end{equation}
For $z=(z_1,\ldots,z_d)\in\T^d$, define the Bohr lift
\begin{equation}\label{eq:bohr-lift-general}
F_{P_N}(\sigma,z)
:=
\sum_{n\le N}a_n n^{-\sigma}z^{\mathbf v(n)},
\qquad
z^{\mathbf m}:=z_1^{m_1}\cdots z_d^{m_d}.
\end{equation}
Then
\begin{equation}\label{eq:vertical-restriction}
F_{P_N}\bigl(\sigma,(p_1^{-it},\ldots,p_d^{-it})\bigr)
=P_N(\sigma+it).
\end{equation}

The elementary arithmetic fact behind the full-torus model is worth isolating.

\begin{proposition}[Prime-frequency injectivity]\label{prop:prime-injectivity}
Let
\[
\omega_N=(\log p_1,\ldots,\log p_d).
\]
Then the resonance lattice
\[
\Lambda_N:=\{\mathbf k\in\Z^d:\langle\mathbf k,\omega_N\rangle=0\}
\]
is trivial.  Consequently the Kronecker orbit
\[
t\longmapsto(p_1^{-it},\ldots,p_d^{-it})
\]
is dense in $\T^d$, and the map $n\mapsto\mathbf v(n)$ is injective.
\end{proposition}

\begin{proof}
If $\mathbf k=(k_1,\ldots,k_d)\in\Lambda_N$, then
\[
0=\sum_{j=1}^dk_j\log p_j
=\log\left(\prod_{j=1}^dp_j^{k_j}\right).
\]
Hence $\prod_jp_j^{k_j}=1$.  Unique factorization in $\Z$ forces $k_j=0$ for every $j$, so $\Lambda_N=\{0\}$.  The continuous Kronecker--Weyl theorem then implies that the orbit is uniformly distributed, and in particular dense, in the full torus; see Haviland and Wintner \cite{HavilandWintner1934}.  Jessen and Tornehave use this finite-integral-base torus representation in their spatial-extension construction \cite[\S116]{JessenTornehave1945}.  Finally, $\mathbf v(m)=\mathbf v(n)$ implies $m=n$ by unique factorization.
\end{proof}

\subsection{The exact Bohr--Jessen identity}

At a zero of $F_{P_N}(\sigma,\cdot)$ one has $\log|F_{P_N}|=-\infty$.  Thus the logarithm is neither continuous nor bounded on the torus, so the continuous Kronecker--Weyl theorem invoked above cannot be applied to it directly.  We first record the required logarithmic integrability.

\begin{lemma}\label{lem:log-integrable}
If $Q$ is a nonzero complex polynomial in $d$ variables, then
\[
\log|Q|\in L^1(\T^d).
\]
\end{lemma}

\begin{proof}
This is the standard finiteness statement underlying logarithmic Mahler measure; compare Boyd \cite[Section~1]{Boyd1998}.  We include the argument because the integrability near the zero set will be used below.  Multiplication by a monomial does not change the modulus on $\T^d$, so it is enough to consider an ordinary polynomial.

For $d=1$, write
\[
Q(z)=c\prod_{j=1}^r(z-\alpha_j).
\]
Since $Q$ is continuous on $\T$, the positive part $\log^+|Q|$ is bounded.  It remains to check each factor in the negative part.\\
 If $|\alpha|\neq1$, then
\[
\inf_{\theta\in[0,2\pi]}|e^{i\theta}-\alpha|>0,
\]
so $\log|e^{i\theta}-\alpha|$ is bounded.\\
If $|\alpha|=1$, write $\alpha=e^{i\theta_0}$.  Then
\begin{equation}\label{eq:circle-root-distance}
|e^{i\theta}-e^{i\theta_0}|
=2\left|\sin\frac{\theta-\theta_0}{2}\right|.
\end{equation}
For $|\theta-\theta_0|\le1$, the elementary bounds $c_1|u|\le|\sin u|\le|u|$ give
\[
c_2|\theta-\theta_0|
\le |e^{i\theta}-e^{i\theta_0}|
\le |\theta-\theta_0|.
\]
Consequently the only possible singularity is logarithmic, and
\begin{equation}\label{eq:log-integral-elementary}
\int_0^\delta |\log u|\dd u
=\delta(1-\log\delta)<\infty
\qquad(0<\delta<1).
\end{equation}
Hence $\log|Q|\in L^1(\T)$.

Assume now that the assertion holds in $d-1$ variables.  Regard $Q$
as a polynomial in the last variable $z_d$, with coefficients depending
polynomially on the remaining variables $z'=(z_1,\ldots,z_{d-1})$:
\[
Q(z',z_d)=\sum_{j=0}^m Q_j(z')z_d^j,
\qquad z'\in\T^{d-1},
\]
where $m=\deg_{z_d}Q$, each $Q_j\in\C[z_1,\ldots,z_{d-1}]$, and
$Q_m\not\equiv0$.  By the induction hypothesis, $\log|Q_m|\in L^1(\T^{d-1})$; in particular the zero set of $Q_m$ has Haar measure zero.  For every $z'$ outside that null set, factor
\[
Q(z',z_d)=Q_m(z')\prod_{j=1}^m(z_d-\alpha_j(z')).
\]
Jensen's one-variable formula, in the form recorded for example in Conway \cite[Chapter~XI, Section~1]{Conway1978}, gives
\begin{align}
L(z')
&:=\int_{\T}\log|Q(z',z_d)|\dd m(z_d)\notag\\
&=\log|Q_m(z')|+\sum_{j=1}^m\log^+|\alpha_j(z')|
\ge \log|Q_m(z')|.
\label{eq:Mahler-fiber-Jensen}
\end{align}
On the other hand, if $M:=\max_{\T^d}|Q|$, then
\[
P(z'):=\int_{\T}\log^+|Q(z',z_d)|\dd m(z_d)
\le \log^+M.
\]
Since $L=P-N$, where
\[
N(z'):=\int_{\T}\log^-|Q(z',z_d)|\dd m(z_d),
\]
we obtain from \eqref{eq:Mahler-fiber-Jensen}
\[
0\le N(z')
=P(z')-L(z')
\le \log^+M+\log^-|Q_m(z')|.
\]
The right-hand side is integrable on $\T^{d-1}$ by the induction
hypothesis.  Since $\log^+|Q|$ is nonnegative, Tonelli's theorem gives
\[
\begin{aligned}
\int_{\T^d}\log^+|Q(z)|\dd m_d(z)
&=
\int_{\T^{d-1}}
\left(
\int_{\T}\log^+|Q(z',z_d)|\dd m(z_d)
\right)
\dd m_{d-1}(z')\\
&=
\int_{\T^{d-1}}P(z')\dd m_{d-1}(z')\\
&\le \log^+M<\infty.
\end{aligned}
\]
Similarly, Tonelli's theorem applied to $\log^-|Q|$ yields
\[
\begin{aligned}
\int_{\T^d}\log^-|Q(z)|\dd m_d(z)
&=
\int_{\T^{d-1}}N(z')\dd m_{d-1}(z')\\
&\le
\log^+M+
\int_{\T^{d-1}}\log^-|Q_m(z')|
\dd m_{d-1}(z')\\
&<\infty,
\end{aligned}
\]
where the last inequality follows from the induction hypothesis.
Therefore
\[
\begin{aligned}
\int_{\T^d}|\log|Q(z)||\dd m_d(z)
&=
\int_{\T^d}\log^+|Q(z)|\dd m_d(z)
+
\int_{\T^d}\log^-|Q(z)|\dd m_d(z)\\
&<\infty.
\end{aligned}
\]
This completes the induction.
\end{proof}

\begin{proposition}[Bohr--Jessen identity]\label{prop:bohr-jessen}
For every fixed $N$ and every $\sigma\in\R$, the limit
\begin{equation}\label{eq:Jessen-general}
J_{P_N}(\sigma)
=
\lim_{T\to\infty}\frac1{2T}
\int_{-T}^{T}\log|P_N(\sigma+it)|\dd t
\end{equation}
exists and
\begin{equation}\label{eq:bohr-jessen-identity}
J_{P_N}(\sigma)
=
\int_{\T^d}\log|F_{P_N}(\sigma,z)|\dd m_d(z),
\end{equation}
where $m_d$ is normalized Haar measure on $\T^d$.  Moreover $J_{P_N}$ is finite and convex on $\R$.
\end{proposition}

\begin{proof}
The frequencies $-\log n$ are integral combinations of the finite basis $-\log p_1,\ldots,-\log p_d$.  Jessen and Tornehave \cite[Theorem~5]{JessenTornehave1945} show that, for a nonzero analytic almost periodic function, the logarithmic mean exists and is obtained by first truncating the logarithm from below and then letting the truncation level tend to zero.

Fix $\varepsilon>0$ and set
\[
g_{\sigma,\varepsilon}(z)
:=\log\max\{|F_{P_N}(\sigma,z)|,\varepsilon\},
\qquad z\in\T^d.
\]
This is continuous on $\T^d$.  Define
\[
\Phi_N:\R\longrightarrow\T^d,
\qquad
\Phi_N(t):=(p_1^{-it},\ldots,p_d^{-it}).
\]
By \eqref{eq:vertical-restriction},
\[
g_{\sigma,\varepsilon}(\Phi_N(t))
=
\log\max\{|P_N(\sigma+it)|,\varepsilon\}.
\]
Proposition~\ref{prop:prime-injectivity} and the Kronecker--Weyl theorem therefore give
\begin{align}
&\lim_{T\to\infty}\frac1{2T}
\int_{-T}^{T}
\log\max\{|P_N(\sigma+it)|,\varepsilon\}\dd t\notag\\
&\quad=
\lim_{T\to\infty}\frac1{2T}
\int_{-T}^{T}g_{\sigma,\varepsilon}(\Phi_N(t))\dd t\notag\\
&\quad=
\int_{\T^d}g_{\sigma,\varepsilon}(z)\dd m_d(z)\notag\\
&\quad=
\int_{\T^d}
\log\max\{|F_{P_N}(\sigma,z)|,\varepsilon\}\dd m_d(z).
\label{eq:truncated-mean}
\end{align}
This is the finite-integral-base spatial extension used by Jessen and Tornehave \cite[\S116]{JessenTornehave1945}.

For $0<\varepsilon\le1$,
\[
\left|\log\max\{|F_{P_N}|,\varepsilon\}\right|
\le
\log^+\|F_{P_N}\|_\infty+\log^-|F_{P_N}|.
\]
The right-hand side is integrable by Lemma~\ref{lem:log-integrable}.  Hence dominated convergence gives
\[
\lim_{\varepsilon\downarrow0}
\int_{\T^d}
\log\max\{|F_{P_N}(\sigma,z)|,\varepsilon\}\dd m_d(z)
=
\int_{\T^d}\log|F_{P_N}(\sigma,z)|\dd m_d(z).
\]
Combining this with \eqref{eq:truncated-mean} and \cite[Theorem~5]{JessenTornehave1945} proves \eqref{eq:bohr-jessen-identity}. Finiteness follows from Lemma~\ref{lem:log-integrable}, while
convexity follows from Jessen and Tornehave \cite[Theorem~7]{JessenTornehave1945}.
\end{proof}

\subsection{Jessen--Tornehave zero density}

For a convex function $J$, write $J'(x-)$ and $J'(x+)$ for its left and right derivatives.

\begin{proposition}[Vertical zero density]\label{prop:JT-density}
Let $Z(P_N)$ denote the multiset of zeros $\rho=\beta+i\gamma$ of $P_N$.  If $a<b$, then the vertical zero frequency in the strip $a<\Ree s<b$ exists and
\begin{equation}\label{eq:JT-density}
\lim_{T\to\infty}\frac1{2T}
\#\{\rho\in Z(P_N):a<\beta<b,\ |\gamma|<T\}
=
\frac1{2\pi}
\bigl(J_{P_N}'(b-)-J_{P_N}'(a+)\bigr).
\end{equation}
  Equivalently,
\begin{equation}\label{eq:Riesz-measure}
\rho_{P_N}:=\frac1{2\pi}J_{P_N}''
\end{equation}
is the locally finite vertical zero-density measure on the real axis.
\end{proposition}

\begin{proof}
	This is Jessen and Tornehave
	\cite[Theorem~31]{JessenTornehave1945}, applied to the ordinary
	Dirichlet series $P_N$.  Their Jensen function is
	\[
	\varphi(\sigma)
	:=
	\lim_{T\to\infty}\frac1{2T}
	\int_{-T}^{T}\log|P_N(\sigma+it)|\dd t,
	\]
	so in our notation $\varphi=J_{P_N}$.  Theorem~31 gives, for every
	$a<b$,
	\[
	H(a,b)
	=
	\frac1{2\pi}
	\bigl(\varphi'(b-0)-\varphi'(a+0)\bigr),
	\]
	where $H(a,b)$ denotes the relative frequency of zeros in the strip
	$a<\Ree s<b$.  This is exactly \eqref{eq:JT-density}.
\end{proof}

\begin{remark}\label{rem:normalization-check}
For $f(s)=1+e^{-\lambda s}$,
\[
J_f(\sigma)=\max\{0,-\lambda\sigma\},
\qquad
J_f''=\lambda\delta_0.
\]
The zeros are $s=(2m+1)\pi i/\lambda$, whose vertical density in $|\Ims s|<T$ is $\lambda/(2\pi)$.  Thus \eqref{eq:Riesz-measure} has the correct constant.
\end{remark}

\subsection{Elementary bounds and total mass}

\begin{proposition}\label{prop:basic-bounds}
Assume $a_1=1$.  Then for every $\sigma\in\R$,
\begin{equation}\label{eq:basic-bounds}
0\le J_{P_N}(\sigma)
\le
\frac12\log\sum_{n\le N}|a_n|^2n^{-2\sigma}.
\end{equation}
\end{proposition}

\begin{proof}
For the upper bound, use Jensen's inequality for the concave function $x\mapsto\log x$: if $\mu$ is a probability measure and $X>0$ is integrable, then
\[
\int\log X\dd\mu\le\log\int X\dd\mu.
\]
Applying this to $X=|F_{P_N}(\sigma,z)|^2$ and normalized Haar measure gives
\begin{align*}
J_{P_N}(\sigma)
&=\frac12\int_{\T^d}\log|F_{P_N}(\sigma,z)|^2\dd m_d(z)\\
&\le\frac12\log\int_{\T^d}|F_{P_N}(\sigma,z)|^2\dd m_d(z).
\end{align*}
Distinct integers define distinct torus characters by Proposition~\ref{prop:prime-injectivity}.  Their orthogonality therefore yields
\[
\int_{\T^d}|F_{P_N}(\sigma,z)|^2\dd m_d(z)
=
\sum_{n\le N}|a_n|^2n^{-2\sigma}.
\]

For the lower bound, use Jensen's formula in one complex variable; see Conway \cite[Chapter~XI, Section~1]{Conway1978}.  If
\[
Q(z)=c_0+c_1z+\cdots+c_mz^m
=c_m\prod_{j=1}^m(z-\alpha_j),
\]
then
\[
\int_{\T}\log|Q(z)|\dd m(z)
=
\log|c_m|+\sum_{j=1}^m\log^+|\alpha_j|
\ge \log|c_0|,
\]
where the last inequality is the equivalent form obtained from
$|c_0|=|c_m|\prod_j|\alpha_j|$.  Apply this successively to the variables $z_d,z_{d-1},\ldots,z_1$.  Fubini is legitimate by Lemma~\ref{lem:log-integrable}.  At each step the circle mean is bounded below by the logarithm of the modulus of the constant coefficient in that variable, and after all variables have been integrated only the constant monomial remains.  Hence
\[
J_{P_N}(\sigma)
=
\int_{\T^d}\log|F_{P_N}(\sigma,z)|\dd m_d(z)
\ge\log|a_1|=0.
\]
\end{proof}

\begin{proposition}\label{prop:total-mass}
Assume $a_1\neq0$ and let $M_N$ be defined by \eqref{eq:MN-intro}.  Then
\begin{equation}\label{eq:totalmass}
J_{P_N}''(\R)=\log M_N,
\qquad
\rho_{P_N}(\R)=\frac{\log M_N}{2\pi}.
\end{equation}
Consequently $J_{P_N}''/\log M_N$ is a probability measure whenever $M_N>1$.
\end{proposition}

\begin{proof}
As $\sigma\to+\infty$,
\[
F_{P_N}(\sigma,z)=a_1+o(1)
\]
uniformly on $\T^d$, and therefore
\[
J_{P_N}(\sigma)=\log|a_1|+o(1).
\]
As $\sigma\to-\infty$, the highest supported term dominates uniformly:
\[
F_{P_N}(\sigma,z)
=
a_{M_N}M_N^{-\sigma}z^{\mathbf v(M_N)}(1+o(1)),
\]
whence
\[
J_{P_N}(\sigma)
=
\log|a_{M_N}|-\sigma\log M_N+o(1).
\]

Since $J_{P_N}$ is convex, these asymptotics determine the limiting
one-sided slopes.  Indeed, the usual secant-slope inequalities for a
convex function show that
\[
J_{P_N,+}'(\sigma)\longrightarrow0
\qquad(\sigma\to+\infty),
\]
and
\[
J_{P_N,+}'(\sigma)\longrightarrow-\log M_N
\qquad(\sigma\to-\infty).
\]

For a finite convex function, the one-sided derivatives are
nondecreasing \cite[Theorem~24.1]{Rockafellar1970}.  Thus its
distributional second derivative is the positive Lebesgue--Stieltjes
measure associated with the increasing right derivative
$J_{P_N,+}'$, characterized by
\[
J_{P_N}''((a,b])
=
J_{P_N,+}'(b)-J_{P_N,+}'(a);
\]
see, for example, Folland \cite[Section~1.5]{Folland1999}.
Letting $a\to-\infty$ and $b\to+\infty$ therefore gives
\begin{align*}
J_{P_N}''(\R)
&=
\lim_{b\to+\infty}J_{P_N,+}'(b)
-
\lim_{a\to-\infty}J_{P_N,+}'(a)\\
&=
0-(-\log M_N)\\
&=
\log M_N.
\end{align*}
This proves the first identity in \eqref{eq:totalmass}.  The second
identity follows from Proposition~\ref{prop:JT-density}:
\[
\rho_{P_N}(\R)
=
\frac{1}{2\pi}J_{P_N}''(\R)
=
\frac{\log M_N}{2\pi}.
\]
\end{proof}

\section{Translation-uniform anti-concentration}\label{sec:anti}

A \emph{Steinhaus random variable} is, by definition, a random variable distributed according to normalized Haar measure on the unit circle $\T$; see, for example, Kahane \cite[Chapter~1]{Kahane1985}. 

The Bessel function of the first kind of order zero is
\begin{equation}\label{eq:J0-definition}
J_0(u)
:=
\sum_{n=0}^{\infty}\frac{(-1)^n}{(n!)^2}
\left(\frac{u}{2}\right)^{2n}.
\end{equation}
For real $u$ it has the integral representation
\begin{equation}\label{eq:J0-integral}
J_0(u)
=
\frac1{2\pi}\int_0^{2\pi}e^{iu\cos\theta}\dd\theta,
\end{equation}
and the classical large-argument expansion implies $J_0(u)=O(u^{-1/2})$ as $u\to+\infty$; see Watson \cite[\S\S2.11, 2.2 and 7.21]{Watson1944}.

\begin{lemma}[Translated logarithmic anti-concentration]\label{lem:Steinhaus}
Fix $K\ge1$.  There exists $C_K<\infty$ such that the following holds for every $m\ge5$.  Let $Z_1,\ldots,Z_m$ be independent Steinhaus variables and let $b_1,\ldots,b_m>0$ satisfy
\begin{equation}\label{eq:comparable-b}
\frac{\max_jb_j}{\min_jb_j}\le K.
\end{equation}
Set
\begin{equation}\label{eq:S-def}
S^2=\sum_{j=1}^m b_j^2.
\end{equation}
Then, uniformly for every $a\in\C$,
\begin{equation}\label{eq:translated-log-main}
\mathbb E\left(\log\left|a+\sum_{j=1}^m b_jZ_j\right|\right)
\ge
\log S-C_K.
\end{equation}
\end{lemma}

\begin{proof}
Put
\[
c_j=\frac{b_j}{S},
\qquad
X=\sum_{j=1}^m c_jZ_j.
\]
Then $\sum_jc_j^2=1$.  If $b_{\min}=\min_jb_j$ and $b_{\max}=\max_jb_j$, then
\[
mb_{\min}^2\le S^2\le mb_{\max}^2,
\]
and \eqref{eq:comparable-b} gives
\begin{equation}\label{eq:cj-comparable}
\frac{1}{K\sqrt m}\le c_j\le\frac{K}{\sqrt m}
\qquad(1\le j\le m).
\end{equation}

Identify $\C$ with $\R^2$ and let $\mu_X$ denote the probability law of $X$.  We use the Fourier convention
\[
\widehat\mu_X(\xi)
:=\int_{\R^2}e^{i\langle\xi,x\rangle}\dd\mu_X(x)
=\mathbb E\left(e^{i\langle\xi,X\rangle}\right).
\]
If $Z_j=e^{i\Theta_j}$ with $\Theta_j$ uniform on $[0,2\pi)$, rotational invariance allows us to take the direction of $\xi$ to be the first coordinate axis.  Hence, by \eqref{eq:J0-integral},
\begin{align*}
\mathbb E\left(e^{i\langle\xi,c_jZ_j\rangle}\right)
&=\frac1{2\pi}\int_0^{2\pi}
 e^{ic_j|\xi|\cos\theta}\dd\theta\\
&=J_0(c_j|\xi|).
\end{align*}
Independence therefore gives
\begin{equation}\label{eq:char-product}
\widehat\mu_X(\xi)
=\prod_{j=1}^mJ_0(c_j|\xi|).
\end{equation}

From the series \eqref{eq:J0-definition},
\[
J_0(u)=1-\frac{u^2}{4}+O(u^4)
\qquad(u\to0).
\]
Thus there exist absolute constants $u_0\in(0,1]$ and $c_0>0$ such that
\begin{equation}\label{eq:J0-small}
|J_0(u)|\le e^{-c_0u^2}
\qquad(0\le u\le u_0).
\end{equation}
The large-argument expansion quoted above yields an absolute constant $C_0$ such that
\begin{equation}\label{eq:J0-large}
|J_0(u)|\le C_0u^{-1/2}
\qquad(u\ge1).
\end{equation}
Moreover, \eqref{eq:J0-integral} and the triangle inequality give
\[
|J_0(u)|
\le\frac1{2\pi}\int_0^{2\pi}|e^{iu\cos\theta}|\dd\theta
=1.
\]
If $u>0$, equality would force $e^{iu\cos\theta}$ to have constant argument for almost every $\theta$, which is impossible.  Hence
\begin{equation}\label{eq:J0-strict}
|J_0(u)|<1\qquad(u>0).
\end{equation}

Write $r=|\xi|$ and split the radial integral into three regions.

\smallskip
\noindent\emph{Region I: $0\le r\le u_0\sqrt m/K$.}
By \eqref{eq:cj-comparable}, $c_jr\le u_0$, so
\begin{equation}\label{eq:region-I}
|\widehat\mu_X(\xi)|
\le
\exp\left(-c_0r^2\sum_{j=1}^mc_j^2\right)
=e^{-c_0r^2}.
\end{equation}
Consequently
\begin{equation}\label{eq:region-I-integral}
2\pi\int_0^{u_0\sqrt m/K}r|\widehat\mu_X(r)|\dd r
\le
2\pi\int_0^\infty re^{-c_0r^2}\dd r
=\frac{\pi}{c_0}.
\end{equation}

\smallskip
\noindent\emph{Region II: $u_0\sqrt m/K\le r\le RK\sqrt m$.}
Choose
\begin{equation}\label{eq:R-choice}
R>\max\{1,u_0,C_0^2\}.
\end{equation}
Then
\begin{equation}\label{eq:middle-arguments}
\frac{u_0}{K^2}\le c_jr\le RK^2.
\end{equation}
By continuity and \eqref{eq:J0-strict},
\begin{equation}\label{eq:rhoK}
\rho_K
:=
\max_{u\in[u_0/K^2,RK^2]}|J_0(u)|
<1.
\end{equation}
Therefore $|\widehat\mu_X(\xi)|\le\rho_K^m$.  Since the annulus has area at most $\pi R^2K^2m$,
\begin{equation}\label{eq:region-II-integral}
\int_{u_0\sqrt m/K\le|\xi|\le RK\sqrt m}
|\widehat\mu_X(\xi)|\dd\xi
\le \pi R^2K^2m\rho_K^m
\le C_K.
\end{equation}

\smallskip
\noindent\emph{Region III: $r\ge RK\sqrt m$.}
Now $c_jr\ge R\ge1$, so \eqref{eq:J0-large} and \eqref{eq:cj-comparable} yield
\begin{equation}\label{eq:region-III-pointwise}
|\widehat\mu_X(\xi)|
\le
\left(C_0K^{1/2}m^{1/4}r^{-1/2}\right)^m.
\end{equation}
For $m>4$,
\begin{align}
&\int_{RK\sqrt m}^{\infty}
 r\left(C_0K^{1/2}m^{1/4}r^{-1/2}\right)^m\dd r\notag\\
&\quad=
C_0^mK^{m/2}m^{m/4}
\int_{RK\sqrt m}^{\infty}r^{1-m/2}\dd r\notag\\
&\quad=
C_0^mK^{m/2}m^{m/4}
\frac{(RK\sqrt m)^{2-m/2}}{m/2-2}\notag\\
&\quad=
\frac{K^2mR^2}{m/2-2}
\left(\frac{C_0^2}{R}\right)^{m/2}.
\label{eq:region-III-exact}
\end{align}
For $m\ge5$,
\[
\frac{m}{m/2-2}=\frac{2m}{m-4}\le10,
\]
and $C_0^2/R<1$.  Thus the last expression is bounded uniformly in $m$, and multiplication by the polar factor $2\pi$ gives
\begin{equation}\label{eq:region-III-integral}
\int_{|\xi|\ge RK\sqrt m}|\widehat\mu_X(\xi)|\dd\xi\le C_K.
\end{equation}

Combining the three regions,
\begin{equation}\label{eq:L1-characteristic}
\|\widehat\mu_X\|_{L^1(\R^2)}\le C_K.
\end{equation}
The Fourier inversion theorem for finite measures with integrable Fourier transform (see Rudin \cite[Section~1.5.2]{Rudin1962}) now implies that $\mu_X$ is absolutely continuous with respect to planar Lebesgue measure.  Its bounded continuous density is
\begin{equation}\label{eq:density-definition}
f_X(x)
:=
\frac1{(2\pi)^2}
\int_{\R^2}e^{-i\langle\xi,x\rangle}
\widehat\mu_X(\xi)\dd\xi,
\end{equation}
so $\dd\mu_X(x)=f_X(x)\dd x$ and
\begin{equation}\label{eq:density-bound}
\|f_X\|_\infty
\le(2\pi)^{-2}\|\widehat\mu_X\|_1
\le C_K.
\end{equation}
Consequently, uniformly in $w\in\C$ and $0<r\le1$,
\begin{align}
\mathbb P(|X+w|<r)
&=\int_{\{x:|x+w|<r\}}f_X(x)\dd x\notag\\
&\le \pi r^2\|f_X\|_\infty
\le C_Kr^2.
\label{eq:small-ball}
\end{align}

Set $Y:=\log^-|X+w|\ge0$.  Pointwise,
\[
Y=\int_0^\infty\mathbf 1_{\{Y>t\}}\dd t.
\]
Tonelli's theorem therefore gives the tail-integral formula
\begin{align}
\mathbb E(Y)
&=\int_0^\infty\mathbb P(Y>t)\dd t\notag\\
&=\int_0^\infty\mathbb P(|X+w|<e^{-t})\dd t\notag\\
&\le\int_0^\infty\min\{1,C_Ke^{-2t}\}\dd t.
\label{eq:negative-log-tail}
\end{align}
Increasing $C_K$ if necessary, we may assume $C_K\ge1$.  Put
\[
t_K:=\frac12\log C_K.
\]
Then $C_Ke^{-2t}\ge1$ exactly when $0\le t\le t_K$, and therefore
\begin{align}
\int_0^\infty\min\{1,C_Ke^{-2t}\}\dd t
&=\int_0^{t_K}1\dd t
+C_K\int_{t_K}^\infty e^{-2t}\dd t\notag\\
&=t_K+\frac{C_K}{2}e^{-2t_K}\notag\\
&=\frac12\log C_K+\frac12
=:C_K'.
\label{eq:negative-log-bound}
\end{align}
Hence
\[
\mathbb E(Y)\le C_K'.
\]
Since $\log^+|X+w|\ge0$,
\begin{equation}\label{eq:normalized-log-bound}
\mathbb E\bigl(\log|X+w|\bigr)\ge-C_K'.
\end{equation}
Finally,
\[
a+\sum_{j=1}^mb_jZ_j
=S\left(\frac aS+X\right),
\]
so \eqref{eq:normalized-log-bound} with $w=a/S$ proves \eqref{eq:translated-log-main}.
\end{proof}

\begin{remark}\label{rem:translation-important}
The uniformity in the translation $a$ is essential for the application below.  After the other torus coordinates are fixed, the isolated prime-coordinate block is translated by a complex number depending on those fixed coordinates, and the estimate must remain uniform in that translation.
\end{remark}

\section{The abstract isolated-prime criterion}\label{sec:abstract}

Let $P_N(s)=\sum_{n\le N}a_n n^{-s}$ and let
\[
\mathcal Q_N\subset\{p\text{ prime}:N/2<p\le N\}.
\]
If $p\in\mathcal Q_N$ and $n\le N$ is divisible by $p$, then $n=p$.  Thus the Bohr coordinate $z_p$ occurs only in the linear monomial $a_pp^{-\sigma}z_p$, and
\begin{equation}\label{eq:isolated-decomposition}
F_{P_N}(\sigma,z)
=A_N(\sigma,z')
+\sum_{p\in\mathcal Q_N}a_pp^{-\sigma}z_p,
\end{equation}
where $A_N$ is independent of the coordinates indexed by $\mathcal Q_N$.

For a finite set $\mathcal Q$ of prime coordinates, write
\[
\dd m_{\mathcal Q}(z)
:=\prod_{p\in\mathcal Q}\dd m(z_p),
\]
where $m$ is normalized Haar probability measure on $\T$.  If $\mathcal P(N)$ denotes the set of all primes not exceeding $N$, set
\[
\mathcal R_N:=\mathcal P(N)\setminus\mathcal Q_N.
\]
Then the Bohr torus decomposes as
\[
\T^d=\T^{\mathcal R_N}\times\T^{\mathcal Q_N},
\qquad
\dd m_d=\dd m_{\mathcal R_N}\,\dd m_{\mathcal Q_N}.
\]

\begin{proposition}[Isolated-prime lower bound]\label{prop:isolated-lower-general}
Let $I\subset\R$ be compact.  Suppose that there exists $K_I\ge1$ such that
\begin{equation}\label{eq:isolated-comparable-general}
K_I^{-1}
\le
\frac{|a_p|p^{-\sigma}}{|a_q|q^{-\sigma}}
\le K_I
\end{equation}
for every $\sigma\in I$ and every $p,q\in\mathcal Q_N$, and suppose $\#\mathcal Q_N\ge5$.  Then
\begin{equation}\label{eq:isolated-lower-general}
J_{P_N}(\sigma)
\ge
\frac12\log\sum_{p\in\mathcal Q_N}|a_p|^2p^{-2\sigma}
-C_{K_I}
\qquad(\sigma\in I),
\end{equation}
where $C_{K_I}$ is the constant from Lemma~\ref{lem:Steinhaus}.
\end{proposition}

\begin{proof}
	Fix $\sigma\in I$ and fix all torus coordinates not indexed by
	$\mathcal Q_N$.  Denote this collection by
	\[
	z'=(z_p)_{p\in\mathcal R_N}\in\T^{\mathcal R_N}.
	\]
	Then, by \eqref{eq:isolated-decomposition},
	\[
	F_{P_N}(\sigma,z',z_{\mathcal Q_N})
	=
	A_N(\sigma,z')
	+\sum_{p\in\mathcal Q_N}a_pp^{-\sigma}z_p,
	\]
	where, for the fixed values of $\sigma$ and $z'$, the quantity
	$A_N(\sigma,z')\in\C$ is independent of the coordinates
	$z_p$, $p\in\mathcal Q_N$.
	
	Write
	\[
	a_p=|a_p|e^{i\vartheta_p},
	\qquad p\in\mathcal Q_N.
	\]
	Since normalized Haar measure on $\T$ is invariant under rotations,
	the changes of variables
	\[
	z_p\longmapsto e^{i\vartheta_p}z_p,
	\qquad p\in\mathcal Q_N,
	\]
	absorb the phases of the coefficients without changing the inner
	integral.  Hence, for fixed $\sigma$ and $z'$, Lemma~\ref{lem:Steinhaus}
	applies with the complex translation
	\[
	a=A_N(\sigma,z')
	\]
	and positive coefficients
	\[
	b_p(\sigma):=|a_p|p^{-\sigma},
	\qquad p\in\mathcal Q_N.
	\]
	
	By \eqref{eq:isolated-comparable-general},
	\[
	\frac{\max_{p\in\mathcal Q_N}b_p(\sigma)}
	{\min_{p\in\mathcal Q_N}b_p(\sigma)}
	\le K_I.
	\]
	Since $\#\mathcal Q_N\ge5$, Lemma~\ref{lem:Steinhaus} therefore gives,
	uniformly in the translation $A_N(\sigma,z')$,
	\begin{align}
	&\int_{\T^{\mathcal Q_N}}
	\log|F_{P_N}(\sigma,z',z_{\mathcal Q_N})|
	\dd m_{\mathcal Q_N}(z_{\mathcal Q_N})\notag\\
	&\qquad\ge
	\frac12\log
	\sum_{p\in\mathcal Q_N}|a_p|^2p^{-2\sigma}
	-C_{K_I}.
	\label{eq:conditioned-isolated-bound}
	\end{align}
	
	The right-hand side of \eqref{eq:conditioned-isolated-bound} is
	independent of $z'$.  Integrating over $\T^{\mathcal R_N}$ and using
	the product decomposition
	\[
	\dd m_d
	=
	\dd m_{\mathcal R_N}\,\dd m_{\mathcal Q_N},
	\]
	together with Proposition~\ref{prop:bohr-jessen}, we obtain
	\begin{align*}
	J_{P_N}(\sigma)
	&=
	\int_{\T^d}
	\log|F_{P_N}(\sigma,z)|\dd m_d(z)\\
	&=
	\int_{\T^{\mathcal R_N}}
	\left[
	\int_{\T^{\mathcal Q_N}}
	\log|F_{P_N}(\sigma,z',z_{\mathcal Q_N})|
	\dd m_{\mathcal Q_N}(z_{\mathcal Q_N})
	\right]
	\dd m_{\mathcal R_N}(z')\\
	&\ge
	\int_{\T^{\mathcal R_N}}
	\left[
	\frac12\log
	\sum_{p\in\mathcal Q_N}|a_p|^2p^{-2\sigma}
	-C_{K_I}
	\right]
	\dd m_{\mathcal R_N}(z')\\
	&=
	\frac12\log
	\sum_{p\in\mathcal Q_N}|a_p|^2p^{-2\sigma}
	-C_{K_I},
	\end{align*}
	because $m_{\mathcal R_N}(\T^{\mathcal R_N})=1$.
\end{proof}

\begin{proof}[Proof of Theorem~\ref{thm:abstract-main}]
We first record that hypothesis \ref{H2intro} automatically determines the scale of $M_N$.  For all sufficiently large $N$, the inequalities in \eqref{eq:H2comp-intro} force $a_p\neq0$ for every $p\in\mathcal Q_N$.  Since $\#\mathcal Q_N\to\infty$, the set $\mathcal Q_N$ is nonempty for all sufficiently large $N$; choosing any $p\in\mathcal Q_N$ gives
\[
\frac N2<p\le M_N\le N.
\]
Consequently
\begin{equation}\label{eq:MN-growth-automatic}
1-\frac{\log2}{\log N}
<\frac{\log M_N}{\log N}\le1,
\end{equation}
and hence
\begin{equation}\label{eq:MN-growth-intro}
\frac{\log M_N}{\log N}\longrightarrow1.
\end{equation}
Thus no separate growth assumption on $M_N$ is needed.

Set
\[
v_N(\sigma):=\frac{J_{P_N}(\sigma)}{\log N}.
\]
By Proposition~\ref{prop:basic-bounds} and hypothesis \ref{H1intro},
\begin{equation}\label{eq:abstract-upper}
0\le v_N(\sigma)
\le
\frac1{2\log N}\log E_N(\sigma),
\end{equation}
so
\begin{equation}\label{eq:limsup-abstract}
\limsup_{N\to\infty}v_N(\sigma)
\le(\alpha-\sigma)_+.
\end{equation}

Fix $\sigma<\alpha$ and choose a compact interval $I\subset(-\infty,\alpha)$ containing $\sigma$.  Recall from \eqref{eq:isolated-energy-intro} that
\[
B_N(\sigma)=\sum_{p\in\mathcal Q_N}|a_p|^2p^{-2\sigma}.
\]
For all sufficiently large $N$, hypothesis \ref{H2intro} and Proposition~\ref{prop:isolated-lower-general} give
\begin{equation}\label{eq:abstract-lower}
v_N(\sigma)
\ge
\frac1{2\log N}\log B_N(\sigma)
-
\frac{C_{K_I}}{\log N}.
\end{equation}
Since the first term tends to $\alpha-\sigma$ and the second tends to zero,
\[
\liminf_{N\to\infty}v_N(\sigma)\ge\alpha-\sigma.
\]
Together with \eqref{eq:limsup-abstract}, this proves
\begin{equation}\label{eq:pointwise-left}
v_N(\sigma)\longrightarrow\alpha-\sigma
\qquad(\sigma<\alpha).
\end{equation}
For $\sigma\ge\alpha$, \eqref{eq:abstract-upper} and hypothesis \ref{H1intro} give
\begin{equation}\label{eq:pointwise-right}
v_N(\sigma)\longrightarrow0.
\end{equation}
Thus $v_N$ converges pointwise on $\R$ to
\[
V(\sigma):=(\alpha-\sigma)_+.
\]
Each $v_N$ is finite and convex.  By Rockafellar \cite[Theorem~10.8]{Rockafellar1970}, pointwise convergence of finite convex functions to the finite convex function $V$ implies uniform convergence on compact subsets.  This proves \eqref{eq:abstract-potential-intro}.

It remains to pass to the zero-density measures.  By Proposition~\ref{prop:total-mass},
\[
\mu_N:=\frac{1}{\log M_N}J_{P_N}''
\]
is a probability measure.  Let $\varphi\in C_c^\infty(\R)$.  Using the definition of the distributional derivative twice,
\begin{align}
\int_\R\varphi(\sigma)\dd\mu_N(\sigma)
&=\frac1{\log M_N}\langle J_{P_N}'',\varphi\rangle\notag\\
&=\frac1{\log M_N}\langle J_{P_N},\varphi''\rangle\notag\\
&=\frac1{\log M_N}
\int_\R J_{P_N}(\sigma)\varphi''(\sigma)\dd\sigma\notag\\
&=\frac{\log N}{\log M_N}
\int_\R v_N(\sigma)\varphi''(\sigma)\dd\sigma.
\label{eq:testfunction-measure}
\end{align}
The support of $\varphi''$ is compact, so local uniform convergence of $v_N$ and \eqref{eq:MN-growth-intro} imply
\begin{equation}\label{eq:testfunction-limit}
\lim_{N\to\infty}\int_\R\varphi\dd\mu_N
=\int_\R V(\sigma)\varphi''(\sigma)\dd\sigma.
\end{equation}
Now $V(\sigma)=\alpha-\sigma$ for $\sigma<\alpha$ and $V(\sigma)=0$ for $\sigma>\alpha$.  Let
\[
V(\sigma)=(\alpha-\sigma)_+.
\]
For every $\varphi\in C_c^\infty(\R)$,
\[
\langle V'',\varphi\rangle
=
\langle V,\varphi''\rangle
=
\int_{-\infty}^{\alpha}
(\alpha-\sigma)\varphi''(\sigma)\dd\sigma.
\]
Integrating by parts gives
\begin{align*}
\int_{-\infty}^{\alpha}
(\alpha-\sigma)\varphi''(\sigma)\dd\sigma
&=
\left[
(\alpha-\sigma)\varphi'(\sigma)
\right]_{-\infty}^{\alpha}
+
\int_{-\infty}^{\alpha}\varphi'(\sigma)\dd\sigma.
\end{align*}
The boundary term vanishes: at $\sigma=\alpha$ the factor
$\alpha-\sigma$ is zero, while at $-\infty$ one has
$\varphi'(\sigma)=0$ because $\varphi$ has compact support.  Hence
\begin{align*}
\int_{-\infty}^{\alpha}
(\alpha-\sigma)\varphi''(\sigma)\dd\sigma
&=
\int_{-\infty}^{\alpha}\varphi'(\sigma)\dd\sigma\\
&=
\left[\varphi(\sigma)\right]_{-\infty}^{\alpha}\\
&=
\varphi(\alpha),
\end{align*}
again because $\varphi(\sigma)=0$ for all sufficiently negative
$\sigma$.  Therefore
\[
\langle V'',\varphi\rangle=\varphi(\alpha)
=\langle\delta_\alpha,\varphi\rangle,
\]
and hence
\[
V''=\delta_\alpha
\]
in the sense of distributions, and \eqref{eq:testfunction-limit} proves convergence against every test function in $C_c^\infty(\R)$.  To pass to arbitrary $h\in C_c(\R)$, choose $\varphi_j\in C_c^\infty(\R)$ with $\|h-\varphi_j\|_\infty\to0$.  Since both $\mu_N$ and $\delta_\alpha$ are probability measures,
\[
\left|\int(h-\varphi_j)\dd\mu_N\right|
\le\|h-\varphi_j\|_\infty,
\qquad
\left|h(\alpha)-\varphi_j(\alpha)\right|
\le\|h-\varphi_j\|_\infty.
\]
First choose $j$ large and then let $N\to\infty$.  It follows that
\[
\int_\R h\dd\mu_N\longrightarrow h(\alpha)
\qquad(h\in C_c(\R)),
\]
which is vague convergence $\mu_N\to\delta_\alpha$.

For completeness, this vague convergence is in fact weak.  Choose $\psi\in C_c(\R)$ such that $0\le\psi\le1$ and $\psi\equiv1$ on a neighborhood of $\alpha$.  Vague convergence gives
\[
\int_\R\psi\dd\mu_N\longrightarrow\psi(\alpha)=1.
\]
Since each $\mu_N$ is a probability measure,
\[
0\le \int_\R(1-\psi)\dd\mu_N
=1-\int_\R\psi\dd\mu_N\longrightarrow0.
\]
Now let $g$ be bounded and continuous.  The product $g\psi$ belongs to $C_c(\R)$, so vague convergence yields
\[
\int_\R g\psi\dd\mu_N\longrightarrow g(\alpha)\psi(\alpha)=g(\alpha).
\]
Moreover,
\[
\left|\int_\R g(1-\psi)\dd\mu_N\right|
\le \|g\|_\infty\int_\R(1-\psi)\dd\mu_N\longrightarrow0.
\]
Therefore
\[
\int_\R g\dd\mu_N\longrightarrow g(\alpha),
\]
which is precisely weak convergence to $\delta_\alpha$.  This proves \eqref{eq:abstract-measure-intro}.
\end{proof}

\section{The Riemann zeta partial sums}\label{sec:zeta}

We now verify the hypotheses of Theorem~\ref{thm:abstract-main} in the simplest case.

\begin{proof}[Proof of Theorem~\ref{thm:zeta-intro}]
Here $a_n=1$ for $n\le N$, so $M_N=N$ and
\begin{equation}\label{eq:zeta-global-energy}
E_N(\sigma)=\sum_{n\le N}n^{-2\sigma}.
\end{equation}
If $\sigma<1/2$, the integral test (or Euler summation) gives
\begin{equation}\label{eq:zeta-energy-left}
E_N(\sigma)
=
\frac{N^{1-2\sigma}}{1-2\sigma}
+O_\sigma(1+N^{-2\sigma}),
\end{equation}
where the notation $O_\sigma$ means that the implied constant may
depend on $\sigma$, but not on $N$.  Since $\sigma<1/2$,
\[
\frac{1}{N^{1-2\sigma}}\longrightarrow0,
\qquad
\frac{N^{-2\sigma}}{N^{1-2\sigma}}
=\frac1N\longrightarrow0.
\]
Hence
\[
E_N(\sigma)
=
\frac{N^{1-2\sigma}}{1-2\sigma}\bigl(1+o(1)\bigr),
\]
and therefore
\[
\frac{1}{2\log N}\log E_N(\sigma)
\longrightarrow
\frac12-\sigma.
\]
At $\sigma=1/2$,
\[
E_N(1/2)=\log N+O(1),
\]
and for $\sigma>1/2$ the sums are bounded above and below by positive constants.  Hence hypothesis \ref{H1intro} holds with
\[
\alpha=\frac12.
\]

Let
\begin{equation}\label{eq:zeta-QN}
\mathcal P_N:=\{p\text{ prime}:N/2<p\le N\}.
\end{equation}
The prime number theorem gives
\begin{equation}\label{eq:dyadic-PNT}
\#\mathcal P_N
=\pi(N)-\pi(N/2)
=\left(\frac12+o(1)\right)\frac{N}{\log N}.
\end{equation}
Let $I\subset\R$ be compact and put $M_I:=\max_{\tau\in I}|\tau|$.  If $p,q\in(N/2,N]$ and $\sigma\in I$, then
\begin{equation}\label{eq:zeta-comparable}
2^{-M_I}
\le
\frac{p^{-\sigma}}{q^{-\sigma}}
\le
2^{M_I}.
\end{equation}
Thus the corresponding coefficients are uniformly comparable.  Moreover, for every $p\in(N/2,N]$,
\[
2^{-2M_I}N^{-2\sigma}
\le p^{-2\sigma}
\le2^{2M_I}N^{-2\sigma}
\qquad(\sigma\in I).
\]
Consequently
\begin{equation}\label{eq:zeta-isolated-energy-bounds}
2^{-2M_I}N^{-2\sigma}\#\mathcal P_N
\le B_N(\sigma):=\sum_{p\in\mathcal P_N}p^{-2\sigma}
\le2^{2M_I}N^{-2\sigma}\#\mathcal P_N.
\end{equation}
If $I\subset(-\infty,1/2)$, then \eqref{eq:dyadic-PNT} and \eqref{eq:zeta-isolated-energy-bounds} give, uniformly for $\sigma\in I$,
\begin{align*}
\frac1{2\log N}\log B_N(\sigma)
&=-\sigma+\frac1{2\log N}\log\#\mathcal P_N+O_I\left(\frac1{\log N}\right)\\
&\longrightarrow\frac12-\sigma.
\end{align*}
In particular,
\begin{equation}\label{eq:zeta-isolated-energy}
B_N(\sigma)=N^{1-2\sigma+o(1)}
\end{equation}
at exponent level for every fixed $\sigma<1/2$.  Taking the abstract set $\mathcal Q_N$ in Theorem~\ref{thm:abstract-main} to be $\mathcal P_N$, hypothesis \ref{H2intro} follows.  Theorem~\ref{thm:abstract-main} now proves both \eqref{eq:zeta-potential-intro} and \eqref{eq:zeta-measure-intro}.
\end{proof}

Combining Theorem~\ref{thm:zeta-intro} with Proposition~\ref{prop:JT-density} gives the direct density formulation.

\begin{corollary}[Critical-line concentration for $\zeta_N$]\label{cor:zeta-density}
For every $\varepsilon>0$,
\begin{equation}\label{eq:zeta-density-coro}
\lim_{N\to\infty}
\frac{2\pi}{\log N}
\lim_{T\to\infty}\frac1{2T}
\#\left\{
\rho=\beta+i\gamma:\zeta_N(\rho)=0,
\ |\beta-1/2|\ge\varepsilon,
\ |\gamma|<T
\right\}
=0,
\end{equation}
with the usual convention of avoiding boundary lines carrying zero frequency.
\end{corollary}

\begin{remark}\label{rem:zeta-support}
Corollary~\ref{cor:zeta-density} does not imply that all zeros approach $\Ree s=1/2$.  Mora's theorem \cite{Mora2013} shows that, for sufficiently large fixed $N$, real parts of zeros are dense throughout the admissible critical strip.  The two statements are compatible: Mora describes support, whereas \eqref{eq:zeta-density-coro} describes normalized vertical mass.
\end{remark}

\section{Dirichlet \texorpdfstring{$L$}{L}-functions}\label{sec:dirichlet}

Let $\chi$ be a fixed Dirichlet character modulo $q$.  Recall that
\[
|\chi(n)|^2=\mathbf 1_{(n,q)=1}.
\]
We write $\varphi(q)=\#(\Z/q\Z)^\times$ for Euler's totient function and $\mu$ for the Möbius function, defined by $\mu(1)=1$, $\mu(n)=(-1)^r$ if $n$ is a product of $r$ distinct primes, and $\mu(n)=0$ if $n$ is divisible by the square of a prime.

\begin{lemma}\label{lem:dirichlet-energy}
For fixed $q$ and $\sigma\in\R$,
\begin{equation}\label{eq:dirichlet-energy}
\sum_{n\le N}|\chi(n)|^2n^{-2\sigma}
=
\sum_{\substack{n\le N\\(n,q)=1}}n^{-2\sigma}.
\end{equation}
If $\sigma<1/2$, then
\begin{equation}\label{eq:dir-left-asymp}
\sum_{\substack{n\le N\\(n,q)=1}}n^{-2\sigma}
=
\frac{\varphi(q)}{q}\frac{N^{1-2\sigma}}{1-2\sigma}
+o_{q,\sigma}(N^{1-2\sigma}),
\end{equation}
where the notation $o_{q,\sigma}(N^{1-2\sigma})$ means that, for fixed
$q$ and $\sigma$, the ratio of the error term to $N^{1-2\sigma}$
tends to $0$ as $N\to\infty$.\\
At $\sigma=1/2$,
\begin{equation}\label{eq:dir-center-asymp}
\sum_{\substack{n\le N\\(n,q)=1}}\frac1n
=\frac{\varphi(q)}q\log N+O_q(1),
\end{equation}
and for $\sigma>1/2$ the sums remain bounded as $N\to\infty$.
\end{lemma}

\begin{proof}
Inclusion--exclusion gives
\begin{align}
A_q(x)
&:=\sum_{n\le x}\mathbf 1_{(n,q)=1}
=\sum_{d\mid q}\mu(d)\left\lfloor\frac{x}{d}\right\rfloor\notag\\
&=x\sum_{d\mid q}\frac{\mu(d)}d+O\left(\sum_{d\mid q}1\right)
=\frac{\varphi(q)}q x+O_q(1).
\label{eq:coprime-count}
\end{align}
Set $c_q:=\varphi(q)/q$.  For $\sigma\neq0$, Abel summation yields
\begin{equation}\label{eq:abel-dir}
\sum_{\substack{n\le N\\(n,q)=1}}n^{-2\sigma}
=A_q(N)N^{-2\sigma}
+2\sigma\int_1^N A_q(x)x^{-2\sigma-1}\dd x.
\end{equation}
Substituting $A_q(x)=c_qx+O_q(1)$ gives
\begin{align}
\sum_{\substack{n\le N\\(n,q)=1}}n^{-2\sigma}
&=c_qN^{1-2\sigma}
+2\sigma c_q\int_1^Nx^{-2\sigma}\dd x\notag\\
&\quad+O_q(N^{-2\sigma})
+O_{q,\sigma}\left(\int_1^Nx^{-2\sigma-1}\dd x\right).
\label{eq:dir-Abel-expanded}
\end{align}
Suppose first that $\sigma<1/2$ and $\sigma\neq0$.  Since
\[
\int_1^Nx^{-2\sigma}\dd x
=\frac{N^{1-2\sigma}-1}{1-2\sigma},
\]
the two main terms in \eqref{eq:dir-Abel-expanded} combine to
\begin{align*}
c_qN^{1-2\sigma}
+\frac{2\sigma c_q}{1-2\sigma}(N^{1-2\sigma}-1)
&=\frac{c_q}{1-2\sigma}N^{1-2\sigma}
+O_{q,\sigma}(1).
\end{align*}
The remaining terms are $o_{q,\sigma}(N^{1-2\sigma})$: indeed, $N^{-2\sigma}/N^{1-2\sigma}=N^{-1}$, while
\[
\int_1^Nx^{-2\sigma-1}\dd x
=
\begin{cases}
O_\sigma(1),&\sigma>0,\\
O(\log N),&\sigma=0,\\
O_\sigma(N^{-2\sigma}),&\sigma<0,
\end{cases}
\]
and each of these is $o(N^{1-2\sigma})$.  The case $\sigma=0$ follows directly from \eqref{eq:coprime-count}.  This proves \eqref{eq:dir-left-asymp}.

At $\sigma=1/2$, partial summation gives
\begin{align*}
\sum_{\substack{n\le N\\(n,q)=1}}\frac1n
&=\frac{A_q(N)}N+\int_1^N\frac{A_q(x)}{x^2}\dd x\\
&=c_q+O_q(N^{-1})
+c_q\int_1^N\frac{\dd x}{x}
+O_q\left(\int_1^N\frac{\dd x}{x^2}\right)\\
&=c_q\log N+O_q(1),
\end{align*}
which is \eqref{eq:dir-center-asymp}.  Finally, if $\sigma>1/2$,
\[
0\le\sum_{\substack{n\le N\\(n,q)=1}}n^{-2\sigma}
\le\sum_{n=1}^{\infty}n^{-2\sigma}<\infty,
\]
so the partial sums are bounded.
\end{proof}

\begin{proof}[Proof of Theorem~\ref{thm:dirichlet-intro}]
Lemma~\ref{lem:dirichlet-energy} verifies hypothesis \ref{H1intro} with $\alpha=1/2$.

Let $\mathcal P_N$ be the dyadic prime set defined in \eqref{eq:zeta-QN}.  For $N>2q$ and $p\in\mathcal P_N$, one has $p>q$ and hence $p\nmid q$, so $|\chi(p)|=1$.  Write $\chi(p)=e^{i\vartheta_p}$.  Since the rotation $z_p\mapsto e^{i\vartheta_p}z_p$ preserves normalized Haar measure on $\T$, the block
\[
\sum_{p\in\mathcal P_N}\chi(p)p^{-\sigma}z_p
\]
has the same distribution under Haar measure as
\[
\sum_{p\in\mathcal P_N}p^{-\sigma}z_p.
\]
Consequently, for $p,r\in\mathcal P_N$ and $\sigma$ in a compact interval $I$,
\[
\frac{|\chi(p)|p^{-\sigma}}{|\chi(r)|r^{-\sigma}}
=\frac{p^{-\sigma}}{r^{-\sigma}}
\le2^{\max_{\tau\in I}|\tau|},
\]
and
\[
\sum_{p\in\mathcal P_N}|\chi(p)|^2p^{-2\sigma}
=\sum_{p\in\mathcal P_N}p^{-2\sigma}
=N^{1-2\sigma+o(1)}
\]
locally uniformly for $\sigma<1/2$, by the zeta calculation.  Taking the abstract set $\mathcal Q_N$ to be $\mathcal P_N$ verifies hypothesis \ref{H2intro}.

Among any $q$ consecutive integers there is an integer congruent to $1$ modulo $q$, hence coprime to $q$.  Therefore
\begin{equation}\label{eq:MN-dir}
N-q+1\le M_{N,\chi}\le N,
\end{equation}
and
\[
\frac{\log M_{N,\chi}}{\log N}\longrightarrow1.
\]
Theorem~\ref{thm:abstract-main} now gives \eqref{eq:dir-potential-intro} and \eqref{eq:dir-measure-intro}.
\end{proof}

\begin{remark}\label{rem:principal-character}
The argument makes no distinction between principal and nonprincipal characters.  In particular, the pole of the limiting principal $L$-function at $s=1$ has no effect on the vertical mean-motion exponent of its finite truncations.
\end{remark}

\section{Holomorphic modular forms}\label{sec:modular}

Let $f$ be the fixed primitive holomorphic Hecke eigenform introduced in Section~\ref{sec:main-results}: it has even weight $k\ge2$, fixed level $Q$, trivial Dirichlet character, and no complex multiplication.  Its normalized coefficients are $\lambda_f(n)=a_f(n)n^{-(k-1)/2}$.

\subsection{Rankin--Selberg quadratic energy}

The classical Rankin--Selberg theorem of Rankin \cite{Rankin1939} and Selberg \cite{Selberg1940} gives
\begin{equation}\label{eq:rankin-selberg}
A_f(x):=\sum_{n\le x}|\lambda_f(n)|^2
=c_fx+O_f(x^{3/5}),
\qquad c_f>0.
\end{equation}

\begin{lemma}[Weighted Rankin--Selberg energy]\label{lem:modular-energy}
For fixed $\sigma<1/2$,
\begin{equation}\label{eq:weighted-RS-left}
\sum_{n\le N}|\lambda_f(n)|^2n^{-2\sigma}
=
\frac{c_f}{1-2\sigma}N^{1-2\sigma}
+o_{f,\sigma}(N^{1-2\sigma}).
\end{equation}
At $\sigma=1/2$,
\begin{equation}\label{eq:weighted-RS-center}
\sum_{n\le N}\frac{|\lambda_f(n)|^2}{n}
=c_f\log N+O_f(1),
\end{equation}
and for $\sigma>1/2$ the weighted partial sums remain bounded.
\end{lemma}

\begin{proof}
For $\sigma\neq0$, Abel summation gives
\begin{equation}\label{eq:modular-Abel}
\sum_{n\le N}|\lambda_f(n)|^2n^{-2\sigma}
=A_f(N)N^{-2\sigma}
+2\sigma\int_1^NA_f(x)x^{-2\sigma-1}\dd x.
\end{equation}
Insert \eqref{eq:rankin-selberg}.  The contribution of the main term $c_fx$ is
\begin{align*}
c_fN^{1-2\sigma}
+2\sigma c_f\int_1^Nx^{-2\sigma}\dd x
&=c_fN^{1-2\sigma}
+\frac{2\sigma c_f}{1-2\sigma}(N^{1-2\sigma}-1)\\
&=\frac{c_f}{1-2\sigma}N^{1-2\sigma}
+O_{f,\sigma}(1).
\end{align*}
The contribution of the error $O_f(x^{3/5})$ is
\begin{equation}\label{eq:RS-weighted-error}
O_f(N^{3/5-2\sigma})
+O_{f,\sigma}\left(\int_1^Nx^{-2/5-2\sigma}\dd x\right).
\end{equation}
The first term divided by $N^{1-2\sigma}$ is $O_f(N^{-2/5})$.  For the integral term there are three cases:
\[
\int_1^Nx^{-2/5-2\sigma}\dd x
=
\begin{cases}
O_{\sigma}(N^{3/5-2\sigma}),&\sigma<3/10,\\
O(\log N),&\sigma=3/10,\\
O_{\sigma}(1),&3/10<\sigma<1/2.
\end{cases}
\]
After division by $N^{1-2\sigma}$ these are respectively $O(N^{-2/5})$, $O(N^{-2/5}\log N)$, and $O(N^{-(1-2\sigma)})$, all tending to zero.  This proves \eqref{eq:weighted-RS-left}; the case $\sigma=0$ follows directly from \eqref{eq:rankin-selberg}.

At $\sigma=1/2$, partial summation yields
\begin{align*}
\sum_{n\le N}\frac{|\lambda_f(n)|^2}{n}
&=\frac{A_f(N)}N+\int_1^N\frac{A_f(x)}{x^2}\dd x\\
&=c_f+O_f(N^{-2/5})
+c_f\int_1^N\frac{\dd x}{x}
+O_f\left(\int_1^Nx^{-7/5}\dd x\right)\\
&=c_f\log N+O_f(1),
\end{align*}
proving \eqref{eq:weighted-RS-center}.  Finally, if $\sigma>1/2$, then $A_f(x)\ll_fx$, and the same partial-summation formula shows that the infinite series $\sum_{n\ge1}|\lambda_f(n)|^2n^{-2\sigma}$ converges.
\end{proof}

It follows from Lemma~\ref{lem:modular-energy} that
\begin{equation}\label{eq:modular-H1}
\frac1{2\log N}
\log\sum_{n\le N}|\lambda_f(n)|^2n^{-2\sigma}
\longrightarrow
\left(\frac12-\sigma\right)_+.
\end{equation}

\subsection{Sato--Tate and isolated prime coordinates}

Let $p\nmid Q$ be prime.  The Ramanujan--Petersson theorem for holomorphic forms, proved by Deligne \cite{Deligne1974}, gives $|a_f(p)|\le2p^{(k-1)/2}$.  Because the Dirichlet character is trivial, the normalized Hecke eigenvalue $\lambda_f(p)$ is real.  Hence there is an angle $\theta_p\in[0,\pi]$ such that
\begin{equation}\label{eq:ST-param}
\lambda_f(p)=2\cos\theta_p.
\end{equation}

Since $f$ has no complex multiplication, the Sato--Tate theorem \cite[Corollary~7.1.7]{BLGG2011} states that the angles $\theta_p$, over unramified primes, are equidistributed on $[0,\pi]$ with respect to
\begin{equation}\label{eq:ST-measure}
\dd\mu_{\mathrm{ST}}(\theta)
=\frac2\pi\sin^2\theta\dd\theta.
\end{equation}
Equivalently, the values $\lambda_f(p)=2\cos\theta_p$ are distributed on $[-2,2]$ according to the semicircle measure
\begin{equation}\label{eq:ST-semicircle}
\dd\nu_{\mathrm{ST}}(x)
=\frac1{2\pi}\sqrt{4-x^2}\dd x.
\end{equation}
Let
\[
\mathcal A:=\{x\in[-2,2]:1\le|x|\le2\},
\qquad
\eta:=\nu_{\mathrm{ST}}(\mathcal A)>0.
\]
The boundary of $\mathcal A$ has $\nu_{\mathrm{ST}}$-measure zero.  Therefore \cite[Corollary~7.1.7]{BLGG2011} gives
\begin{equation}\label{eq:ST-count-global}
\#\{p\le x:p\nmid Q,\ 1\le|\lambda_f(p)|\le2\}
=\eta\,\pi(x)+o(\pi(x)).
\end{equation}
Applying this at $x=N$ and $x=N/2$, and using the prime number theorem, yields
\begin{equation}\label{eq:ST-count-dyadic}
\#\{N/2<p\le N:p\nmid Q,\ 1\le|\lambda_f(p)|\le2\}
=\left(\frac{\eta}{2}+o(1)\right)\frac{N}{\log N}.
\end{equation}
Define
\begin{equation}\label{eq:modular-PN}
\mathcal P_{N,f}
:=
\{p:N/2<p\le N,\ p\nmid Q,\ 1\le|\lambda_f(p)|\le2\}.
\end{equation}
Then $\#\mathcal P_{N,f}\to\infty$.

Let $I\subset\R$ be compact and put $M_I:=\max_{\tau\in I}|\tau|$.  If $p,r\in\mathcal P_{N,f}$ and $\sigma\in I$, then $1\le|\lambda_f(p)|,|\lambda_f(r)|\le2$ and $1/2<p/r<2$.  Hence
\begin{equation}\label{eq:modular-comparable}
\frac{1}{2^{M_I+1}}
\le
\frac{|\lambda_f(p)|p^{-\sigma}}{|\lambda_f(r)|r^{-\sigma}}
\le
2^{M_I+1}.
\end{equation}
Thus the isolated weights are uniformly comparable.  Moreover, for $\sigma<1/2$,
\begin{align}
B_{N,f}(\sigma)
&:=\sum_{p\in\mathcal P_{N,f}}|\lambda_f(p)|^2p^{-2\sigma}.
\label{eq:modular-isolated-energy-def}
\end{align}
Since $1\le|\lambda_f(p)|^2\le4$ and $N/2<p\le N$,
\[
2^{-2|\sigma|}N^{-2\sigma}\#\mathcal P_{N,f}
\ll_{\sigma} B_{N,f}(\sigma)
\ll_{\sigma}N^{-2\sigma}\#\mathcal P_{N,f}.
\]
Together with \eqref{eq:ST-count-dyadic}, this gives
\begin{equation}\label{eq:modular-isolated-energy}
B_{N,f}(\sigma)=N^{1-2\sigma+o(1)},
\end{equation}
with the $o(1)$ uniform for $\sigma$ in compact subsets of $(-\infty,1/2)$.  Hence
\begin{equation}\label{eq:modular-H2}
\frac1{2\log N}\log B_{N,f}(\sigma)
\longrightarrow\frac12-\sigma
\end{equation}
locally uniformly on $(-\infty,1/2)$.

\subsection{Proof of the modular theorem}

\begin{proof}[Proof of Theorem~\ref{thm:modular-intro}]
Equation \eqref{eq:modular-H1} verifies hypothesis \ref{H1intro} with $\alpha=1/2$.  Taking the abstract set $\mathcal Q_N$ to be $\mathcal P_{N,f}$, equations \eqref{eq:modular-comparable} and \eqref{eq:modular-H2} verify hypothesis \ref{H2intro}.

By \eqref{eq:ST-count-dyadic}, for all sufficiently large $N$ there is a prime $p\in(N/2,N]$ with $\lambda_f(p)\neq0$.  Consequently
\[
N/2<M_{N,f}\le N,
\qquad
\frac{\log M_{N,f}}{\log N}\longrightarrow1.
\]
Theorem~\ref{thm:abstract-main} now gives \eqref{eq:modular-potential-intro} and \eqref{eq:modular-measure-intro}.

Finally, from $a_f(n)=\lambda_f(n)n^{(k-1)/2}$,
\begin{align}
\mathcal L_N(s,f)
&=\sum_{n\le N}a_f(n)n^{-s}\notag\\
&=\sum_{n\le N}\lambda_f(n)n^{-(s-(k-1)/2)}\notag\\
&=L_N\left(s-\frac{k-1}{2},f\right).
\label{eq:classical-shift}
\end{align}
Thus the real parts of all zeros are translated by $(k-1)/2$, and the concentration line $\Ree w=1/2$ in the normalized variable $w=s-(k-1)/2$ becomes
\[
\Ree s=\frac12+\frac{k-1}{2}=\frac{k}{2}.
\]
\end{proof}

\begin{corollary}\label{cor:level-one}
In particular, Theorem~\ref{thm:modular-intro} applies to every normalized holomorphic Hecke cusp eigenform for $\mathrm{SL}_2(\Z)$.  Zero-density concentration occurs on $\Ree s=1/2$ for the normalized Hecke-eigenvalue truncations and on $\Ree s=k/2$ for the classical Fourier-coefficient truncations.
\end{corollary}

\section{Conclusion}

We have proved a general criterion that converts quadratic coefficient growth, together with a sufficiently large family of isolated prime coordinates, into concentration of normalized vertical zero density on a single line.  The Bohr--Jessen identity turns the vertical logarithmic mean into a torus Mahler measure, while the translation-uniform Steinhaus estimate prevents the isolated linear block from spending excessive Haar measure near zero.  The resulting lower bound matches the elementary $L^2$ upper bound at logarithmic scale.

The limiting potential is
\[
V_\alpha(\sigma)=(\alpha-\sigma)_+,
\]
whose unique corner gives $V_\alpha''=\delta_\alpha$.  For the partial sums of the Riemann zeta function and of fixed Dirichlet $L$-functions, $\alpha=1/2$.  Rankin--Selberg and Sato--Tate give the same value for normalized holomorphic Hecke eigenforms of fixed level without complex multiplication; in the classical Fourier-coefficient normalization the corresponding line is $\Ree s=k/2$.\\
The present argument is intrinsically formulated in the Jessen
mean-motion regime, with $T\to\infty$ for each fixed $N$ before
$N\to\infty$.  A coupled $(N,T)$ limit would require quantitative
control of the passage from vertical averages to Haar averages that
is uniform in the growing frequency set
$\{\log p:p\le N\}$, and no such uniform estimate is used here.\\

The conclusion is compatible with the much wider support of individual zeros for a fixed truncation: concentration of normalized vertical density is a measure-theoretic statement and does not assert geometric concentration of the entire zero set.

\end{document}